\documentclass[11pt]{birkjour}
\newtheorem{thm}{Theorem}[section]

\newtheorem{lem}[thm]{Lemma}
\newtheorem{prop}[thm]{Proposition}
\theoremstyle{definition}

\newtheorem{nota}[thm]{Notation}
\theoremstyle{remark}
\newtheorem{rem}[thm]{Remark}
\newtheorem{ex}[thm]{Example}
\numberwithin{equation}{section}
\begin{document}
 \author[Ahmed Jafar ,  Omar Ajebbar and Elhoucien Elqorachi ]{Ahmed Jafar $^{1}$, Omar Ajebbar $^{2}$ and Elhoucien Elqorachi $^{1}$}
	
	\address{%
		$^1$
	 Ibn Zohr University, Faculty of sciences,
		Department of mathematics,
		Agadir,
		Morocco}
	\address{%
	$^2$
	Sultan Moulay Slimane University, Multidisciplinary Faculty,
	Department of mathematics and computer science,
	Beni Mellal,
	Morocco}
\email{hamadabenali2@gmail.com, omar-ajb@hotmail.com, elqorachi@hotmail.com }
	 	
	 	\title[ A Kannappan-sine subtraction law on semigroups ]{A Kannappan-sine subtraction law on semigroups}
	 	
	 	\begin{abstract}
		  Let  $S$ be a semigroup, $z_0$ a
		  fixed element in   $S$ and  $\sigma:S \longrightarrow S$ an involutive automorphism.
		We determine the  complex-valued solutions of Kannappan-sine subtraction law $f(x\sigma(y)z_0)=f(x)g(y)-f(y)g(x),\; x,y \in S
		$. As an application we solve  the following  variant of  Kannappan-sine subtraction law  viz. $f(x\sigma(y)z_0)=f(x)g(y)-f(y)g(x)+\lambda g(x\sigma(y)z_0) ,\; x,y \in S,$ where   $\lambda \in \mathbb{C}^{*}
			$. The continuous solutions on topological semigroups are given and an example to illustrate the main results is also given.
	 \end{abstract}
	\maketitle\emph{}
	\section{Introduction}
Throughout this paper  $S$ is a semigrpoup, $z_0$ is a fixed element in $S$ and  $\sigma:S\longrightarrow S$ is an  involutive automorphism. That $\sigma$ is involutive means that $\sigma\circ\sigma(x)=x$.

The subtraction law for the sine
$$\sin(x-y)=\sin(x)\cos(y)-\sin(y)\cos(x),\;\; x,y \in \mathbb{R},$$
gives rise to functional equations on any semigroup  $S$ as follows
\begin{equation}
\label{7}
f(x\sigma(y))=f(x)g(y)-f(y)g(x)
, \;x,y \in S,
\end{equation}
for unknown functions $f,g:S \longrightarrow \mathbb{C}$. The functional equation (\ref{7}) has been solved on groups, see for example \cite{ebb}, \cite[Theorem 4.12]{g}, monoids generated by their  squares \cite[Proposition 3.6]{eb}, semigroups generated by their squares \cite[Theorem 5.2]{ajj} and  general monoids \cite[Theorem 4.2]{c}. The most current results about (\ref{7})  on general semigroups are \cite[Theorem 4.2]{vb} and
\cite[Proposition 3.2]{q}.

For more additional discussions about this functional equation and its history see \cite[Ch.13]{a}, \cite[Theorem II.2]{ebb} and their references. Thus the subtraction law for the sine  (\ref{7}) is solved in large generality,  and it makes sense to solve other trigonometric functional equations on semigroups by expessing their solutions in terms of solutions of (\ref{7}). This is what we shall do for Kannappan-Sine subtraction law (\ref{6}) (defined below).\\\\In  \cite[Proposition 16]{i} he solved the functional equation
\begin{equation}\label{4}
f(xyz_{0})=f(x)f(y),\,x,y \in S,
\end{equation}
on semigroups, and where $z_{0}$ is a fixed element in $ S$. The solutions of (\ref{4}) are of the form $f=\chi(z_0)\chi,$ where $\chi:$ $S\to \mathbb{C}$ is a multiplicative function.
\\In the present paper we consider the following Kannappan-Sine subtraction law
\begin{equation}
\label{6}
f(x\sigma(y)z_0)=f(x)g(y)-f(y)g(x), \;x,y \in S.
\end{equation} If $S$ is a monoid, then by replacing $y$ by the identity element in (\ref{6}), equation (\ref{6}) can be written as follows: $$ \alpha f(x\sigma(y))=f(x)g(y)-f(y)g(x),\,x,y\in S$$ or $$f(x\sigma(y))=f(x)g(y)-f(y)g(x)+\beta g(x\sigma(y)),\,x,y\in S.$$
Furthermore, the second equation is also closely related to the sine subtraction law
\begin{equation}\label{ast}
F(x\sigma(y))=F(x)g(y)-F(y)g(x), \;x,y \in G,
\end{equation} where $F=\dfrac{1}{\alpha}f -g$,  and then explicit formulas
for $f$ and $g$ on groups exist in the literature (see for example  \cite[Corollary 4.4]{l}).\\\\The purpose of the present paper is to show  how the relation between (\ref{6}) and equation (\ref{ast}) on monoids extends to the much wider framework of (\ref{6}) on semigroups. The above computations used on monoids  does not apply, and must provide a separate exposition. We show that any solutions of (\ref{6}) are closely related to the solutions of (\ref{ast}), and so can be expressed in terms of multiplicative and additive functions of a semigroup  $S.$\\\\Kannappan functional equations are studied in several works. We refer for example to \cite{boui}, \cite{kel}, \cite{f}, \cite{ff}, \cite{s} and \cite{pe}.\\\\Our main contributions to the knowledge about  Kannappan-Sine subtraction law   (\ref{6}) are the following\\
- We extend the setting from groups to semigroups with involution.\\- We relate the solutions of equation (\ref{6}) to those of equation (\ref{ast}), and we derive formulas for solutions of (\ref{6}) (Section 5, Theorem 5.1)\\-As an application,   we determine the solutions   of the following variant of  Kannappan-sine subtraction law (\ref{6}), namely
  \begin{equation}
  \label{2}
  f(x\sigma(y)z_0)=f(x)g(y)-f(y)g(x)+\lambda g(x\sigma(y
  )z_0), \;x,y \in S ,
  \end{equation}
 where $\lambda \in\mathbb{C}^{*}$ (Section 6, Theorem 6.1), and we apply the theory to an example (Section 7).
 	\section{Set up, notation and terminology}
		Throughout this paper  we enforce the set up below.\\
	A   semigroup 	$S$ is   a set with an associative composition rule.	If $X$ is a topological space we denote by $C(X)$ the algebra of continuous functions from $X$ to the field of complex numbers  $\mathbb{C}$.
		
		A map $A:S\longrightarrow \mathbb{C}$ is said to be additive if	$
			A(xy)=A(x)+A(y)$   for   all $  x,y \in S$,
		and a function $\chi :S\longrightarrow \mathbb{C}$ is multiplicative if
		$
			\chi(xy)=\chi(x)\chi(y)$ for  all $ x,y \in S.$
	 	If $\chi  $ is multiplicative and  $\chi \neq 0$
		then  we call  $ \chi$ an \textit{exponential}. For a multiplicative function  $\chi :S\longrightarrow \mathbb{C}$,  we define the \textit{null space} $I_{\chi}$ by
		\begin{center}
			$I_{\chi} := \{x\in S\hspace{0.1cm} | \hspace{0.1cm} \chi(x) = 0\}$.
		\end{center}
		Then $I_{\chi}$ is  either empty or a proper subset of $S$ and $I_{\chi}$ is a
		two sided ideal in S if not empty and $S\setminus I_{\chi}$ is a subsemigroup of $S$.
		For any  subset $T \subseteq S$,  let  $T^2 :=\{xy  \hspace{0.1cm}| \hspace{0.1cm}x,y   \in T\}$,  and
		$T^2z_0:=\{xyz_0 \hspace{0.1cm}|\hspace{0.1cm} x,y\in T  \}$.
		
	In order to characterize   some solutions of our functional equations, we use the Ebanks's partition \cite{c,d} of the null space $I_{\chi}$ into the disjoint union $I_{\chi} =  P_{\chi}\cup (I_{\chi} \setminus P_{\chi} ) $  with
		\begin{equation*}
		P_{\chi}:=\{p\in I_{\chi}\setminus I_{\chi}^2 \hspace{0.1cm} | \hspace{0.1cm}up,pv,upv \in  I_{\chi}\setminus I_{\chi}^2 \hspace{0.1cm}\text{for}\hspace{0.1cm} \text{all} \hspace{0.1cm} u,v \in S\setminus I_{\chi}\}.
		\end{equation*}
	 	  Lemma \ref{eb} follows directly from the definition of  $P_{\chi}$.
	 	\begin{lem}
	 	 	\label{eb}  $p\in P_{\chi} \implies up,pv, upv \in  P_{\chi}$ for all  $u,v \in
	 	 	S\setminus I_{\chi}\ $.
	 	 \end{lem}
	 	   Throughout this paper we will use Lemma \ref{sigma} without mentioning explicit.
		 \begin{lem}
		 	\label{sigma}
		 	Let $S$ be a semigroup, $\sigma  :S\longmapsto S$ is an involutive automorphism and $\chi :S\longmapsto\mathbb{C}$ be a multiplicative function such that $\chi\circ\sigma =\chi$. Then \\
		 	(a) If $x\in   I_{\chi}$, then $\sigma(x)\in   I_{\chi}$.\\
		 	(b) If $x\in I_{\chi}^2$, then $\sigma(x)\in I_{\chi}^2$. \\
		 	(c) If $x\in 	S\setminus 	I_{\chi}$, then $\sigma(x)\in	S \setminus 	I_{\chi}$.\\
		 	(d) If $x\in 	P_{\chi}$, then $\sigma(x)\in 	P_{\chi}$.\\
		 	(e) If $x\in 	I_{\chi}\setminus 	P_{\chi}$, then $\sigma(x)\in 	I_{\chi}\setminus 	P_{\chi}$.
		 \end{lem}
		 \begin{proof}
		 	(a)-(c): See \cite[Lemma 4.1]{c}. (d)-(e): See \cite[Lemma 3.2]{ass}
		 \end{proof}
	
	 	\begin{lem}
			\label{ind}
			Let $S$ be  a semigroup, $n\in \textit{N},$
			and $ \chi_1,\chi_2,...\chi_n:S\longrightarrow \mathbb{C}$    be   exponential functions. Then  \\
			(a) $\{\chi_1,\chi_2,...,\chi_n\}$ is linearly independent.  \\
			(b) If  $A: S\setminus I_{\chi}\longrightarrow \mathbb{C}$ is  a non-zero additive function and $\chi$ is exponential on $S$, then    the set   $\{\chi A, \chi\}$ is linearly independent on   $S\setminus I_{\chi}$.
		\end{lem}
		 \begin{proof}
		 	(a): See \cite [Theorem 3.18]{g}. (b): See \cite[Lemma 4.4]{ajjj}.
		 \end{proof}
		 For convenience   we introduce the following notation.
		 \begin{nota}
		 	\label{not}
		 	Let $\Psi_{A\chi,\rho}:S\longrightarrow\mathbb{C}$ denote a function $f$ of the form of \cite[Theorem 3.1 (B)]{d},  where  $\chi:S\longrightarrow
		 	\mathbb{C}$ is  an exponential   function, $A:S\setminus I_{\chi}\longrightarrow
		 	\mathbb{C}$ is additive,  $\rho:P_{\chi}\longrightarrow
		 	\mathbb{C}$ is the restriction of $f$,   and conditions (i)  and (iI) hold.
		 \end{nota}
		
		 In the present paper we will use frequently  \cite[Theorem 4.2]{vb} in the proof of our main results. The exponential $m$ and the function $\phi_m$ in \cite[Theorem 4.2]{vb} are replaced by $\chi$ and $\Psi_{A\chi,\rho}$ respectively.
			\section{Auxiliary results}
		In this section we solve  the functional equation
		\begin{equation}
		\label{04}
		f(x\sigma(y)z_0)=f(x)f(y), \; x,y \in S,
		\end{equation}
		which  has been solved on semigroups   by Stetk\ae r \cite[Proposition 16]{i} for the particular case $\sigma=Id$.
	 	\begin{prop}
	 		\label{p1}
	 	Let  $f  :S\longmapsto\mathbb{C}$  be a solution     of the functional equation (\ref{04}). Then we have the following  \\
	 	(i) 	$f(z_0)\neq0 \Leftrightarrow f\ne0.$\\
	  (ii)	$f=\chi(z_0)\chi,$	
	 	  where    $\chi:S\longrightarrow \mathbb{C}$ is a  multiplicative function such that  	$\chi\circ \sigma =\chi$.
	 	\end{prop}
	 	\begin{proof}
	 			Let $x,y,z\in S$.
	 		(i)	If $f=0$ it is clear that   $f(z_0)=0$. Conversely suppose that $f(z_0)=0$.
	 	 By using (\ref{04}) we get that
	 		\begin{equation*}
	 		\label{88}
	 		f(xz_0\sigma(xz_0)z_0)= f (xz_0)^2
	 		= f(xz_0\sigma(x)) f(z_0)=0,
	 		\end{equation*}
	 		which gives
	 		$  f (xz_0)=0.$
	 		Therefore using this and (\ref{04}) we obtain
	 		$$f(x\sigma(y)z_0)=f(x)f(y)=0.$$
	 		This  implies that $f=0$. \\
	 (ii)	The case $f=0$ is trivial, so from now on we my assume that $f\neq0$. By using  again (\ref{04}) we find that
	 		\begin{equation*}
	 		\label{02}
	 		f(x \sigma(yzz_0)z_0)= f (x\sigma(y))f(zz_0)
	 		= f(x \sigma(yz)) f(z_0).
	 		\end{equation*}
	 		Replacing $z$ by $\sigma(z_0)$  in the last equation  we  obtain
	 		\begin{equation}
	 		\label{09}
	 		f(x\sigma(y))f(\sigma(z_0)z_0)= f(x \sigma(y)z_0) f(z_0)=f(x)f(y)f(z_0).
	 		\end{equation}
	 		Since $f\neq 0$ by assumption, we have $f(z_0)\neq0$ by Proposition \ref{p1}(i) and then  the
	 		right hand side of (\ref{09}) is non-zero. Therefore we get  $f(\sigma(z_0)z_0)\neq0$. Next  dividing  identity (\ref{09}) by $f(\sigma(z_0)z_0)$ we get
	 		\begin{equation}
	 		\label{333}
	 		f(x\sigma(y) )=\beta f(x)f(y),
	 		\end{equation}
	 		where $\beta:= \dfrac{f(z_0)}{f(\sigma(z_0)z_0)}\neq0$. Now by using  (\ref{333}) we get 	\begin{equation*}
	 		f(x\sigma(yz) )= \beta f (x )f(yz)
	 		= \beta f(x\sigma(y) ) f(z )
	 		= \beta^2 f(x)f(y)f(z).
	 		\end{equation*}
	 		Since $f\neq0$ the previous identities imply that $f(yz)= \beta f(y)f(z).$
	 		This shows that    $\beta f=:\chi $, where $\chi  :S\longmapsto\mathbb{C}$ is a multiplicative function (moreover, it is an exponential  since $f\neq0$ and $\beta \neq0$), so $f=\chi/\beta$.  By using (\ref{333}) again  we have
	 		\begin{equation*}
	 		\chi(x)\chi\circ\sigma(y)= \chi(x\sigma(y))=\chi(x)\chi(y),
	 		\end{equation*}
	 		which implies that
	 		$\chi\circ \sigma =\chi$ because $\chi\neq0$.
	 	Finally by using (\ref{04}) we have
	 		\begin{equation*}
	 		\dfrac{1}{\beta}[\chi(x)\chi\circ \sigma(y)  \chi(z_0)] 	= \dfrac{\chi(x\sigma(y)z_0)}{\beta}= \dfrac{\chi(x) }{\beta}\dfrac{\chi(y) }{\beta},
	 		\end{equation*}
	 		which implies that $\dfrac{1}{\beta}=\chi(z_0)$. Hence $f=\chi(z_0) \chi$.
	 		The converse statement is evident.
	 	\end{proof}
		\section{ Preparatory works for Kannappan-sine subtraction law }
			In this section we prove some useful lemmas which will be used  in the proof of our main results (Theorem \ref{t1}).
			
		  \begin{lem}
		 						  	\label{uu}
		 						 Let $f, g :S\longmapsto\mathbb{C}$ be a solution  of the functional equation (\ref{6}) such that $f $ and  $g $ are linearly independent. Then $f(z_0)=0$.
		 						  \end{lem}
		 						  \begin{proof}
		 						  	 By contradiction, suppose that $f(z_0)\neq0$.	Let $x,y,z\in S$ be arbitrary. By  using (\ref{6}) we get
		 						  		\begin{equation*}
		 						  	 f((x\sigma(yzz_0)z_0  )  =  f(x\sigma(yz ))g(z_0)-f(z_0)g(x\sigma(yz ))
		 						  			= f(x)g(yzz_0)-f(yzz_0)g(x).
		 						  		\end{equation*}
		 						  		Replacing $z $ by $\sigma(z_0)$ in   the previous identities we get
		 						  		
		 						  		  $f(x\sigma(y  )z_0)g(z_0)-f(z_0)g(x\sigma(y )z_0)$
		 						  		    \begin{eqnarray*}
		 						  		  &=& g(z_0)[f(x)g(y)- f(y)g(x)]-f(z_0)g(x\sigma(y )z_0)\\
		 						  		 &=& f(x)g(y\sigma(z_0)z_0)-f(y\sigma(z_0)z_0)g(x)\\     &	=& f(x)g(y\sigma(z_0)z_0)-g(x)[f(y)g(z_0)-f(z_0)g(y)],
		 						  		\end{eqnarray*}
		 						  		and then we deduce that
		 						  		\begin{equation}
		 						  		\label{n3}
		 						  		-f(z_0)[g(x\sigma(y )z_0)+g(x)g(y)]=f(x)[g(y\sigma(z_0)z_0)-g(y)g(z_0)].
		 						  		\end{equation}
		 						  	 Since $f(z_0)\neq0$, dividing equation (\ref{n3}) by $f(z_0)$ we obtain
		 						  		\begin{equation}
		 						  		\label{n4}
		 						  		g(x\sigma(y )z_0)= -g(x)g(y)+f(x)\phi(y),
		 						  		\end{equation}
		 						  		where
		 						  		\begin{equation*}
		 						  		\phi(y):= -\dfrac {g(y\sigma(z_0)z_0)-g(y)g(z_0)} {f(z_0)}.
		 						  		\end{equation*}
		 						  		Now, substituting (\ref{n4}) into (\ref{n3})  we get
		 						  		
		 						  		$-f(z_0)[-g(x)g(y)+f(x)\phi(y)+g(x)g(y)]$
		 						  			\begin{equation*}
		 						  			=f(x)[-g(y)g(z_0)+f(y)\phi(z_0)-g(y)g(z_0)],
		 						  			\end{equation*}
		 						  		 which gives
		 						  		$
		 						  		f(z_0)f(x)\phi(y)=f(x)[ 2g(y)g(z_0)-f(y)\phi(z_0)].
		 						  		$
		 						  	Since $ f\neq0$ because $f(z_0)\neq0$, we obtain
		 						  			$\phi(y)=a f(y)+bg(y)$, where
		 						  			    $a:=-\dfrac{\phi(z_0)}{f(z_0) }$ \;and\; $b:=\dfrac{2g(z_0)}{f(z_0) }$,
		 						  		and then by using this we obtain $$a =-\dfrac{\phi(z_0)}{f(z_0) }=-\dfrac{a f(z_0)+bg(z_0)}{f(z_0) }=-a-\dfrac{b^2}{2},$$ which implies that $a=- {b^2}/{4}$. Therefore
$\phi(y)=-\dfrac{b^2}{4} f(y)+bg(y).$ This gives a new form to (\ref{n4}) as follows
		 						  		\begin{equation}
		 						  		\label{132}
		 						  		g(x\sigma(y )z_0)= -g(x)g(y)-\dfrac{b^2}{4} f(x)f(y)+bf(x)g(y) .
		 						  		\end{equation}
		 						  	  Now, by using the system (\ref{6}) and (\ref{132}) we obtain
		 						  			 	\begin{equation}
		 						  			 	\label{n5}
		 						  			 	  ( g-\dfrac{b}{2} f) (x\sigma(y)z_0) =-(g-\dfrac{b}{2} f) (x ) (g-\dfrac{b}{2} f)
(y).		 						  			 	\end{equation}
		 						  			 	   		 						  			 	   By putting $x=y=z_0$ in (\ref{n3}) and in (\ref{6}) we get respectively
		 						  			 	   \begin{equation}
		 						  			 	   \label{ss}
		 						  			 	     g(z_0\sigma(z_0)z_0)=0,
		 						  			 	     \end{equation}    and
		 						  			 	    \begin{equation}
		 						  			 	  \label{24}  f(z_0\sigma(z_0)z_0)= f(z_0)g(z_0)-f(z_0)g(z_0)=0.
		 						  			 	   \end{equation}
		 						  			 	    Next,
		 						  			 	  applying Proposition \ref{p1} to  identity (\ref{n5})  we infer that there exists  a multiplicative  function  $\chi$ on $S$ such that $ g-\dfrac{b}{2} f=-\chi(z_0)\chi$  with     $\chi\circ \sigma=\chi$.  Moreover,  we have     $\chi(z_0)\neq0$ (so $\chi$ is an exponential) since $g-\dfrac{b}{2} f\neq0$, because  $f$ and $g$ are linearly independent.  Finally putting $x=y=z_0$ in (\ref{n5})  and using (\ref{ss})  and (\ref{24}) we obtain
		 						  			 	   $$   (g-\dfrac{b}{2} f)(z_0\sigma(z_0)z_0)=-\chi(z_0)^4=0,$$
		 						  			  which is a contradiction  because $\chi(z_0)\neq0$.  Thus     $f(z_0)=0$.
		 						  	\end{proof}
		 						  	\begin{rem}
		 						  	 If $S$ is a   group (or a monoid) and $f,g$ is a solution of (\ref{6}), then we can check easily that  $f(z_0)=0$. Indeed,   putting  $x=y=e$ in (\ref{6})  we get $f(z_0)= f(e)g(e)-f(e)g(e)=0$. So working with  an identity element or a group inversion gives  advantages and sometimes makes   calculations simpler. This justifies why the proof of Lemma \ref{uu} is long in the absence of identity element $e$.
		 						  	\end{rem}
		 						 \begin{lem}
		 						 		\label{r8}
		 						 		Let    $f, g :S\longmapsto\mathbb{C}$ be a solution  of the functional equation (\ref{6}) such that $f$ and $g$ are linearly independent. Then   for all $x,y \in S$ we have
		 						 		\begin{equation}
		 						 		\label{08}
		 						 		g(\sigma(z_0)z_0 )f(x\sigma(y))= g(z_0)[f(x)g(y)-f(y)g(x)]+f(\sigma(z_0)z_0)g(x\sigma(y)).
		 						 		\end{equation}
		 						 \end{lem}
		 						 \begin{proof}
		 					 	Let $x,y,z,t\in$  $S$ be arbitrary. Using (\ref{6}) we have
		 						 \begin{equation*}
		 						 	f(x\sigma(yz)\sigma(t) z_0) =f(x\sigma(yz) )g( t)-f(t)g(x\sigma(yz))
		 						 \end{equation*}
		 						 and
		 						 \begin{equation*}
		 						 	f(x\sigma(y)\sigma(zt) z_0) =	f(x\sigma(y))g(zt)- f(zt)g(x\sigma(y)).
		 						 \end{equation*}
		 						 Since 	$f(x\sigma(yz)\sigma(t) z_0)  =f(x\sigma(y)\sigma(zt) z_0)$	we deduce that
		 						 \begin{equation}
		 						 	\label{0.1}
		 						 	f(x\sigma(y))g(zt)- f(zt)g(x\sigma(y))=f(x\sigma(yz) )g( t)-f(t)g(x\sigma(yz)).
		 						 \end{equation}
		 						 By putting $z =\sigma(z_0)$ and $t = z_0$  in (\ref{0.1}), and then using the fact that $ f(z_0)=0$ (Lemma \ref{uu})		 we get
		 						
		 						 $f(x\sigma(y))g(\sigma(z_0) z_0)- f(\sigma(z_0) z_0)g(x\sigma(y))$	
		 						 \begin{eqnarray*}
		 						 	&=&f(x\sigma(y)z_0)g(z_0) -f(z_0 )g(x\sigma(y)z_0) \\
		 						 	&=&   g(z_0)[f(x)g(y)-f(y)g(x)],
		 						 \end{eqnarray*}
		 						 which implies that
		 						 \begin{equation*}
		 						 	g(\sigma(z_0)z_0 )f(x\sigma(y))= g(z_0)[f(x)g(y)- f(y)g(x)]+f(\sigma(z_0)z_0)g(x\sigma(y)).
		 						 \end{equation*}
		 						
		 						 	\end{proof}	  	
		 						 \begin{lem}
		 						 	\label{r10}
		 						 	Let   $f, g :S\longmapsto\mathbb{C}$ be a solution  of the functional equation (\ref{6}) such that     $ f $  and $g$ are linearly independent.   Then we have the following \\
		 						 	(i)  $g(z_0)\neq0$ \\
		 						 	(ii ) If $g(\sigma(z_0)z_0)=0
		 						 	$ then $ f(\sigma(z_0)z_0)\neq0$
		 						   \end{lem}
		 						    \begin{proof} By Lemma \ref{uu} we have $f(z_0)=0$.
		 						    	(i)   By  a contradiction,  we suppose that $g(z_0)=0$. Making    the substitutions  $(x,\sigma(yz_0) )$ and  $(x\sigma(y),z_0 )$ in (\ref{6}) we get   that
		 						    	\begin{eqnarray*}
		 						    		f(x\sigma(yz_0)z_0)&= & f(x )g(yz_0)-f(yz_0)g(x) \\
		 						    		& =& f(x\sigma(y))g(z_0)-f(z_0)g(x\sigma(y))=0.
		 						    	\end{eqnarray*}
		 						    	It follows  that
		 						    	\begin{equation*}
		 						    	f(x )g(yz_0)-f(yz_0)g(x)=0, \; x\in S.
		 						    	\end{equation*}
		 						    	Since $f$ and $g$ are linearly independent we obtain that  $g(yz_0)=f(yz_0)=0$ for all $y\in S$. Then by using     (\ref{6}) we get that
		 						    	\begin{equation*}
		 						    	f(x\sigma(y)z_0)=f(x)g(y)-f(y)g(x)=0,\; x,y\in S,
		 						    	\end{equation*}
		 						    	which implies that
		 						    	\begin{equation*}
		 						    	\label{03}
		 						    	f(x)g(y)=f(y)g(x), \; x,y\in S.
		 						    	\end{equation*}
		 						    	This contradict the fact that  $f$ and $g$ are linearly independent. So we conclude that $g(z_0)\neq0$.\\
		 						    	(ii)  Assume, for the sake of contradiction, that  $f(\sigma(z_0)z_0 )=0$. Since $g(\sigma(z_0)z_0 )=0$ by hypothesis,  equation (\ref{08})  reduces to
		 						    	\begin{equation*}
		 						    	g(z_0) f(x)g(y)-g(z_0)f(y)g(x)=0, \; x,y\in S,
		 						    	\end{equation*}
		 						    	which becomes
		 						    	\begin{equation*}
		 						    	f(x)g(y)=  f(y)g(x), \; x,y\in S,
		 						    	\end{equation*}
		 						    	since
		 						    	$g(z_0)\neq0$ by (i).   This   contradicts that $f $ and $g$ are linearly independent.  Therefore $f(\sigma(z_0)z_0) \neq0$.
		 						   	\end{proof}
		 						   	\section{Main results}
		 						   	Now we are ready to find out the solutions of Kannappan-sine subtraction law (\ref{6}) on semigroups.
		 						   	 \begin{thm}
		 						   	\label{t1}
		 						     The solutions $f,g :S\longrightarrow \mathbb{C}$ of the  Kannappan-sine subtraction law (\ref{6}) can be listed as follows.\\
		 						   	(1) $f=0$ and $g$ arbitrary.\\
		 						   	(2) $S\neq S^2z_0 $, $f$ is any non-zero function such that $f(S^2z_0)=0$  and $g=\kappa f$, where $\kappa\in \mathbb{C}$ is a constant.
		 						   	 \\
		 						   	(3) There exist  constants $\gamma, b \in \mathbb{C}^{*}, c  \in \mathbb{C} \setminus \{\pm 1\} $ and an exponential function $\chi$ on
		 						   	 $S$  with $\chi \neq \chi\circ\sigma$, $\chi(z_0)=\dfrac{-2b}{1+c}$ and $\chi\circ\sigma(z_0)=\dfrac{  2b}{1-c}$ such that
		 						   	\begin{equation*}
		 						   	f=  \dfrac{\chi+\chi\circ\sigma}{2\gamma} +c\dfrac{\chi-\chi\circ\sigma}{2\gamma}   \hspace{0.6cm} and \hspace{0.6cm}  g=b (\chi-\chi\circ\sigma ).
		 						   	\end{equation*}
		 						   	(4) There exist  constants $\beta, b \in \mathbb{C}^{*}, c  \in \mathbb{C}   $ and an exponential function $\chi$ on
		 						   	$S$ with
		 						   	$\chi \neq \chi\circ\sigma$ and $\chi(z_0)=  \chi\circ\sigma(z_0)=1/\beta$ such that
		 						   	\begin{equation*}
		 						   	f=b
		 						    (\chi-\chi\circ\sigma) \hspace{0.4cm} and \hspace{0.4cm}  g=  \dfrac{\chi+\chi\circ\sigma}{2\beta}+c\dfrac{\chi-\chi\circ\sigma}{2\beta}.
		 						   	\end{equation*}
		 						   	(5) There exist  constants $\alpha,\delta, b  \in \mathbb{C}^{*}, c\in \mathbb{C}  $    and an exponential  function
		 						   		$\chi$  on	$S$
		 						   		with  $ \alpha \neq\pm (2b \delta +\alpha c)$,  $\chi \neq \chi\circ\sigma$, $	\chi (z_0)=	\dfrac {2b }{\alpha (1+c)+2b\delta   } \;\; and \;\; 	 \chi\circ \sigma(z_0)=\dfrac {2b }{\alpha (c-1)+2b\delta   }$
		 						   	 such that
		 						   	\begin{equation*}
		 						    f= \alpha   \dfrac{\chi+\chi\circ\sigma}{2\delta }+  (\alpha c+2b\delta) \dfrac{\chi-\chi\circ\sigma}{2\delta }   \hspace{0.2cm} and \hspace{0.2cm}  g=   \dfrac{\chi+\chi\circ\sigma}{2\delta} +c  \dfrac{\chi-\chi\circ\sigma}{2\delta}.
		 						   	\end{equation*}
		 						   	(6) There exist  constants $\gamma,  c\in \mathbb{C}^{*} $, an exponential function $\chi$ on $S$, a non-zero
		 						   		additive
		 						   	function $A: S \setminus  I_{\chi}\longrightarrow \mathbb{C}$ and a function $\rho: P_{\chi}\longrightarrow \mathbb{C}$ with $\chi\circ\sigma=\chi, A \circ\sigma=$
		 						   		$-A, \rho\circ\sigma=-\rho $ and
		 						   	$\chi(z_0)=A(z_0)=\dfrac{-1}{c}$ such that
		 						   	\begin{equation*}
		 						   	f= \gamma\left( \chi+c \Psi_{A\chi,\rho}\right)  \hspace{0.4cm} and \hspace{0.4cm} g= \Psi_{A\chi,\rho}.
		 						   	\end{equation*}
		 						   	(7) There exist  constants $ \alpha,  c \in \mathbb{C}$ and $ \delta \in \mathbb{C}^{*}$, an exponential function $\chi$ on $S$,
		 						   a non-zero
		 						   	additive
		 						   	function $A: S \setminus  I_{\chi}\longrightarrow \mathbb{C}$ and a function $\rho: P_{\chi}\longrightarrow \mathbb{C}$ with $\alpha c+\delta\neq0$, $\chi\circ\sigma$
		 						   $=\chi,  A \circ\sigma=-A,  \rho\circ\sigma=-\rho $,
		 						   	$\chi(z_0)=\dfrac{1}{\alpha c+\delta} $   and  $ A(z_0)=-\dfrac{\alpha}{\alpha c+\delta }$ such that
		 						   	\begin{equation*}
		 						   	f=   \dfrac{1}{\delta} [  \alpha\chi+ (\alpha c +\delta) \Psi_{A\chi,\rho}] \hspace{0.4cm} and \hspace{0.4cm} g=   \dfrac{1}{\delta} \left( \chi+ c \Psi_{A\chi,\rho}\right).
		 						   	\end{equation*}
		 						   		Moreover if $S$ is a topological semigroup and $f,g \in C(S)$, then $ \chi, \chi\circ\sigma \in C(S), A \in C(S \setminus  I_{\chi})$ and $\rho \in C(P_{\chi})$.
		 						   	 \end{thm}
		 							 \begin{proof}
		 							 	 If $f=0$ then $g$ can be chosen arbitrary, so this solution occurs in part (1). For the rest of the proof we assume that  $f\neq0$.
		 							 	 If $f$ and $g$ are linearly dependent, then there exists a constant $\kappa \in \mathbb{C}$ such that $g=\kappa f$ and then Eq. (\ref{6}) reduces to $f(x\sigma(y)z_0)=0$ for all $x,y \in S$. Then we must have $S \neq S^2z_0$ since $f \neq0$ by assumption. This proves the solution family (2). From now on we  may assume that $f$ and $g$ are linearly independent. Now by Lemma \ref{r10} (i) we have $g(z_0)\neq0$.
		 							 	 	
		 							 	 \textbf{Case 1}: Suppose $g(\sigma(z_0)z_0 )=0$, then by Lemma \ref{r10} (ii) we get  $f(\sigma(z_0)z_0 )\neq0$ and therefore Eq.(\ref{08}) can be rewritten as
		 							 	 \begin{equation}
		 							 	 \label{r88}
		 							 	 g(x\sigma(y))= \gamma g(x)f(y)- \gamma f(x)g(y),  \; x,y \in S,
		 							 	 \end{equation}
		 							 	 with $\gamma:=\dfrac{g(z_0)}{f(\sigma(z_0)z_0 )} \neq0$. From  (\ref{r88}) we read that the pair $(g, \gamma f)$ satisfies the sine subtraction law (\ref{7}), so according to \cite[Theorem 4.2]{vb}  and taking into account that $f$ and $g$ are linearly independent,  we infer that we have only  the following possibilities
		 							 	
		 							 	 	(i)  There exist constants $b\in \mathbb{C}^{*}, c \in \mathbb{C}$ and  an exponential function $\chi $ on $S$, with $\chi \circ \sigma \neq \chi$  such that
		 							 	 	\begin{equation*}
		 							 	 	  g= b(\chi-\chi\circ\sigma) \;\;and \;\; \gamma f=\dfrac{\chi +\chi\circ\sigma}{2 }+c  \dfrac{\chi-\chi\circ\sigma}{2 }.
		 							 	 	\end{equation*}
		 							 	 	By using (\ref{6}) we get  after some calculation that
		 							 	 	\begin{equation*}
		 							 	 \left( 	(1+c)\chi(z_0)+2 b\right) \chi(xy)+\left( (1-c )\chi\circ\sigma(z_0)-2 b\right) \chi\circ \sigma(xy)=0,\; x, y \in S,
		 							 	 	\end{equation*}
		 							 	 	which gives by Lemma \ref{ind} (a) that
		 							 	 		\begin{equation}
		 							 	 		\label{rrr}
		 							 	 	  	(1+c)\chi(z_0)=-2 b\;\;\;  and\; \;\;  	(1-c )\chi\circ\sigma(z_0)=2 b.  \end{equation}
		 							 	 	  	As $b \neq0$ we get from (\ref{rrr}) that $c \neq \pm1$. So
		 							 	 	  		\begin{equation*}
		 							 	 	  	 \chi(z_0)=\dfrac{ -2b}{1+c }\;\;\;  and \;\;\;    \chi\circ\sigma(z_0)=\dfrac{ 2b}{1-c },  \end{equation*}
		 							 	 	  	and then we obtain solution family (3).
		 							 	 	  	
		 							 	 	  	 (ii) There exist a  constant $c\in \mathbb{C}$, an exponential function $\chi $ on $S$, an
		 							 	 	  	 additive
		 							 	 	  	 function $A: S \setminus  I_{\chi}\longrightarrow \mathbb{C}$ and a function $\rho: P_{\chi}\longrightarrow \mathbb{C}$ such that
		 							 	 	  	 \begin{equation*}
		 							 	 	  	   g=   \Psi_{A\chi,\rho}\hspace{0.4cm} and \hspace{0.4cm}  \gamma f= \chi +c\Psi_{A\chi,\rho},
		 							 	 	  	 \end{equation*}
		 							 	 	with $\Psi_{A\chi,\rho}\neq0$,  $\chi \circ \sigma = \chi$, $A\circ \sigma = -A$ and $\rho\circ\sigma=-\rho$. Now,  we explore the location of $z_0$. Since  $0\neq g(z_0)=\Psi_{A\chi,\rho}(z_0)$, we have    $ z_0 \in P_{\chi}\cup (S \setminus I_{\chi} )$  because $\Psi_{A\chi,\rho}=0$ on $I_{\chi}\setminus  P_{\chi}$. If  $ z_0 \in P_{\chi}$, then
		 							 	 	 $$0\neq g(z_0)=\Psi_{A\chi,\rho}(z_0)= \rho(z_0)  \; \; and $$
		 							 	 	  $$ 0=\gamma f(z_0)=\chi(z_0)+c \Psi_{A\chi,\rho}(z_0)=c \rho(z_0),$$
		 							 	 	 since $\chi(z_0)=0$, and $f(z_0)=0$ by Lemma \ref{uu}.  Then from two last identities we  get   $c=0$.   Therefore $$\gamma f(\sigma(z_0)z_0)=\chi(\sigma(z_0)z_0)=\chi(\sigma(z_0)z_0)=\chi( z_0)^2=0,$$ which contradicts that $f(\sigma(z_0)z_0)\neq0$ in the present case.
		 							 	 	  Thus, we conclude that  $ z_0 \in   S \setminus I_{\chi} $,
		 							 	 	and by  a  small computations based on (\ref{6}),   we  can check easily that  for all $x,y \in  S \setminus I_{\chi} $ we have
		 							 	 	\begin{equation}
		 							 	 	\label{r25}
		 							 	 	\left(1+c A(z_0) \right)\chi(z_0)\chi(xy)+\left(c\chi(z_0)+ 1 \right) \chi(xy)A(xy) =0.
		 							 	 	\end{equation}
		 							 	 	If $A=0$ then we get
		 							 	   $\chi(z_0)\chi(xy)=0$ which contradicts that $\chi(z_0)\neq0$ and $\chi(xy)\neq0$. So $A\neq0$ and from (\ref{r25}) we derive by Lemma \ref{ind} (b) that
		 							 	   $$ \left(1+c A(z_0) \right)\chi(z_0)=0\; \; and\; \;  c\chi(z_0)+ 1 =0,      $$
		 							 	    from which we  read  that $c\neq0$ because $c \chi(z_0)=-1\neq0$. Then we  deduce that $A(z_0)=\chi(z_0)=-1/c$. So, by writing $\gamma$ instead of $\dfrac{1}{\gamma}$, we obtain the solution in part (6).
		 							 	 	
		 							 	 	\textbf{Case 2}: Suppose $g(\sigma(z_0)z_0 )\neq0$.\\
		 							 	 	 \underline{Subcase 2.1}. $f(\sigma(z_0)z_0 )=0$. Then (\ref{08}) reduces to
		 							 	 	 \begin{equation}
		 							 	 	 \label{rtt}
		 							 	 	 f(x\sigma(y))= \beta f(x)g(y)-\beta f(y)g(x),  \; x,y \in S,
		 							 	 	 \end{equation}
		 							 	 	 with $\beta:=\dfrac{g(z_0)}{g(\sigma(z_0)z_0 )} \neq0$. From  (\ref{rtt}) we read that the pair $(f, \beta g)$ satisfies the sine subtraction law (\ref{7}), so according to \cite[Theorem 4.2]{vb}  and taking into account that $f$ and $g$ are linearly independent,  we infer that we have only  the following possibilities :
		 							 	 	
		 							 	 	 (i)  There exist constants $b\in \mathbb{C}^{*}, c \in \mathbb{C}$ and  an exponential function $\chi $ on $S$ with $\chi \circ \sigma \neq \chi$   such that
		 							 	 	 \begin{equation*}
		 							 	 f= b(\chi-\chi\circ\sigma) \hspace{0.4cm}and \hspace{0.4cm} \beta g=\dfrac{\chi +\chi\circ\sigma}{2 }+c  \dfrac{\chi-\chi\circ\sigma}{2 }.
		 							 	 	 \end{equation*}
		 							 	 	 By using (\ref{6}) we get  after simplification that
		 							 	 	 \begin{equation*}
		 							 	 	 \left( \beta\chi\circ \sigma(z_0)-1\right) \chi(xy)+\left( 1-	 \beta\chi (z_0) \right) \chi\circ \sigma(xy)=0,\; x,y \in S,
		 							 	 	 \end{equation*}
		 							 	 	 which implies, by Lemma \ref{ind} (a), that
		 							 	 	 \begin{equation*}
		 							 	 	  \beta\chi\circ \sigma(z_0)-1=0\hspace{0.3cm}  and \hspace{0.3cm}  1-	 \beta\chi (z_0) =0.  \end{equation*}
		 							 	 	This gives $\chi\circ \sigma(z_0)= \chi(z_0)=1/\beta $. Then we are   in solution family (4).
		 							 	 	
		 							 	 	 (ii) There exist a  constant $c\in \mathbb{C}$, an exponential function $\chi $ on $S$, an additive function $A: S \setminus  I_{\chi}\longrightarrow \mathbb{C}$ and a function $\rho: P_{\chi}\longrightarrow \mathbb{C}$ such that
		 							 	 	 \begin{equation*}
		 							 	 	f=   \Psi_{A\chi,\rho}\hspace{0.4cm} and \hspace{0.4cm}  \beta g=\chi +c \Psi_{A\chi,\rho},
		 							 	 	 \end{equation*}
		 							 	 	 with $\Psi_{A\chi,\rho}\neq0$,  $\chi \circ \sigma = \chi$, $A\circ \sigma = -A$ and $\rho\circ\sigma=-\rho$.  Furthermore   $z_0 \in S\setminus I_{\chi} $. Indeed, by Lemma \ref{uu} we have that $0=f(z_0)=\Psi_{A\chi,\rho}(z_0)$,  and  by Lemma \ref{r10}(i) we get that $0\neq \beta g(z_0)=\chi(z_0)+c\Psi_{A\chi,\rho}(z_0)=\chi(z_0)$. By a  small computation based on (\ref{6}),   we have
		 							 	 	 \begin{equation}
		 							 	 	 \label{05}
		 							 	 	    A(z_0)\chi(z_0)  \chi(xy)+\left( \chi(z_0)- \dfrac{1}{\beta} \right)A(xy) \chi(xy) =0
		 							 	 	 \end{equation}
for all $x,y \in S \setminus I_{\chi}$.\\
If $A=0,$ then  $f= 0$ on $S\setminus I_{\chi}$ and necessarily we have $\rho \neq0 $ because $f\neq0$. Next, let  $x\in P_{\chi}$  and $y\in S\setminus I_{\chi}$, then $x\sigma(y)z_0 \in P_{\chi}$  by    Lemma \ref{eb}. Now by  condition (iI) of \cite[Theorem 3.1 (B)]{d} we have
		 							 	 	\begin{equation*}
		 							 	 	\rho(x\sigma(y)z_0)= \rho(x)\chi(\sigma(y)z_0)=\rho(x)\chi(y)\chi(z_0),
		 							 	 	\end{equation*}
		 							 	 and therefore	by using (\ref{6})   we get that
		 							 	 	\begin{eqnarray*}
		 							 	 		\rho(x)\chi(y)\chi( z_0)=\rho(x\sigma(y)z_0) = f(x\sigma(y)z_0) &=& f(x) g(y)-f(y)g(x)
		 							 	 	 \\
		 							 	 	&=& \dfrac{1   }{\beta}\chi(y)\rho(x).
		 							 	 	\end{eqnarray*}
		 							 	 	This implies that
		 							 	     $\chi(z_0)=1/\beta $ because $ \rho\neq0$ and $\chi(y)\neq0$,	
		 							 	 so we have  the special case of  the solution family (7) corresponding to $\alpha=0$.
		 							 	  If $A\neq0
		 							 	 	$ then from (\ref{05}) we get by Lemma \ref{ind} (b) that
		 							 	 	$$   A(z_0)\chi(z_0) =0  \; \; and \; \;  \chi(z_0)- \dfrac{1}{\beta}   =0.$$
		 							 	 	Hence
		 							 	 	 $\chi(z_0)=1/\beta $ and $A(z_0)=0$. So, again we obtain the special case of solution family (7) corresponding to $\alpha=0$.\\
		 							 	 	 \underline{Subcase 2.2}: If $f(\sigma(z_0)z_0 )\neq0$  then Eq. (\ref{08}) can be rewritten as follows
		 							 	 	 	\begin{equation}
		 							 	 	 	\label{rvv}
		 							 	 	 	 f(x\sigma(y))= \delta f(x)g(y)-\delta f(y)g(x)+\alpha   g(x\sigma(y)),  \; x,y \in S,
		 							 	 	 	\end{equation}
		 							 	 	 where $$\delta:=\dfrac{g(z_0)}{g(\sigma(z_0)z_0 )} \neq0 \;\; and \;\; \alpha:=\dfrac{f(\sigma(z_0)z_0)}{g(\sigma(z_0)z_0 )} \neq0.$$
		 							 	 	 Now, we can reformulate the form of Eq. (\ref{rvv})  as follows
		 							 	 	 $$
		 							 	 	 \left(  f-\alpha g\right)(x\sigma(y)))=\delta \left(  f-\alpha g\right)(x)g(y)-\delta  \left(  f-\alpha g\right)(y) g(x),\; x, y \in S
		 							 	 	$$
		 							 	 	where $ x,y \in S$.
		 						From  the last equation we read that the pair $ \left(   f-\alpha g, \delta g\right) $ satisfies the sine subtraction law (\ref{7}).  So, in view of \cite[Theorem 4.2]{vb} and using the assumption that $\{f,g\}$ is linearly independent, we infer that we have the following possibilities :
		 						
		 						 (i)  There exist constants $b\in \mathbb{C}^{*}, c \in \mathbb{C}$ and  an exponential function $\chi $ on $S$, with $\chi \circ \sigma \neq \chi$ such that
		 						 \begin{equation*}
		 						 f-\alpha g=b(\chi -\chi\circ \sigma)\;\;\; and \;\;\; \delta g= \dfrac{\chi +\chi\circ \sigma}{2}+c\dfrac{\chi -\chi\circ \sigma}{2},
		 						 \end{equation*}
		 						 which implies that
		 						 \begin{equation*}
		 						  f= \alpha   \dfrac{\chi+\chi\circ\sigma}{2\delta }+  (\alpha c+2b\delta) \dfrac{\chi-\chi\circ\sigma}{2\delta }   \hspace{0.4cm}and \hspace{0.4cm}    g=\dfrac{\chi+\chi\circ\sigma}{2 \delta}+c  \dfrac{\chi-\chi\circ\sigma}{2\delta }.
		 						 \end{equation*}
		 						 By applying (\ref{6}) to the pair $(x,\sigma(y))$ we get  after some computations that
		 						 \begin{equation*}
		 						 \left( 2 b-(\alpha+\alpha c +2b \delta)\chi (z_0) \right) \chi(xy)-\left( 	2 b+(\alpha-\alpha c -2b \delta) \chi(\sigma(z_0))\right) \chi\circ \sigma(xy)=0,
		 						 \end{equation*}
		 						for all $x,y \in S$, which implies by Lemma \ref{ind} (a) that
		 						 \begin{equation}
		 						 \label{n1}
		 						 	2 b-(\alpha+\alpha c +2b \delta)\chi (z_0) =0\hspace{0.3cm}  and \hspace{0.3cm} 		2 b+(\alpha-\alpha c -2b \delta) \chi\circ \sigma(z_0)=0.  \end{equation}
		 						As $b \neq0$, we deduce from the identities above that $ \alpha \neq\pm (2b \delta +\alpha c)$ and then
		 					\begin{equation*}	
		 					\chi (z_0)=	\dfrac {2b }{\alpha (1+c)+2b\delta   }\hspace{0.3cm} and \hspace{0.3cm}	 \chi\circ \sigma(z_0)=\dfrac {2b }{\alpha (c-1)+2b\delta   }.
		 					\end{equation*}
		 						 This gives the  solution in family (5).
		 						
		 							 (ii) There exist a  constant $c\in \mathbb{C}$, an exponential function $\chi $ on $S$, an additive function $A: S \setminus  I_{\chi}\longrightarrow \mathbb{C}$ and a function $\rho: P_{\chi}\longrightarrow \mathbb{C}$ such that
		 							 \begin{equation*}
		 							   f-\alpha g= \Psi_{A\chi,\rho}\;\;\; and \;\;\;   \delta g=\chi+c   \Psi_{A\chi,\rho}
		 							 \end{equation*}
		 							 with $\chi \circ \sigma = \chi$, $A\circ \sigma = -A$ and $\rho\circ\sigma=-\rho$.
		 							 Then the forms of $f$ and $g$  become  respectively \begin{equation*}
		 							 f =\dfrac{1}{\delta}[ \alpha\chi+ (\alpha c+\delta)\Psi_{A\chi,\rho}]\;\;\; and \;\;\;    g=\dfrac{1}{\delta}(\chi+c   \Psi_{A\chi,\rho}).
		 							 \end{equation*}
		 							  Moreover $z_0 \in S\setminus I_{\chi}$. Indeed, if $z_0\in I_{\chi}\setminus P_{\chi}$ then we have $$    g(z_0)=\dfrac{1}{\delta}(  \chi(z_0)+c\Psi_{A\chi,\rho}(z_0))=0,$$
		 							  since $\chi(z_0)=0$ and $\Psi_{A\chi,\rho}=0$ on  $ I_{\chi}\setminus P_{\chi}$.
		 							   This is a contradiction because $g(z_0)\neq0$  by assumption. If $z_0 \in P_{\chi} $
		 							 then
		 							\begin{equation}
		 							\label{55} 0\neq \dfrac{1}{\delta}(  \chi(z_0)+c\Psi_{A\chi,\rho}(z_0))= \dfrac{c}{\delta}\rho(z_0)    \; \; and \end{equation}
		 							 $$0= f(z_0)=  \dfrac{1}{\delta}[\alpha \chi(z_0)+( \alpha c+\delta)\Psi_{A\chi,\rho}(z_0)]= \dfrac{1}{\delta}  { (\alpha c +\delta)} \rho(z_0),$$
		 							 since $\chi(z_0)=0 $ and $f(z_0)=0$ by Lemma \ref{uu}.  This gives $ \alpha c +\delta=0$ because  $\rho(z_0)\neq0$ by  (\ref{55}).   Therefore  by using this we obtain    $$f(\sigma(z_0)z_0)=\dfrac{1}{\delta}[\alpha\chi(\sigma(z_0)z_0)+( \alpha c+\delta)\Psi_{A\chi,\rho}(\sigma(z_0)z_0)]= \dfrac{\alpha}{\delta}\chi( z_0)^2=0,$$ but  this contradicts the   assumption that $f(\sigma(z_0)z_0)\neq0$ in the present subcase.
		 							   Now for all $x,y\in  S \setminus I_{\chi}$ we have  $xy z_0 \in  S \setminus I_{\chi}$, and then by applying (\ref{6}) to the pair $(x,\sigma(y))$, a  small computation shows that
		 							 \begin{equation}
		 							 \label{r23}
		 							\left( \alpha \chi(z_0)+ (\alpha c+\delta)A(z_0)\chi(z_0) \right) \chi(xy)-\left((\alpha c +\delta)\chi(z_0)  - 1 \right)\chi(xy) A(x y) =0
		 							 \end{equation}
for all $x,y \in S \setminus I_{\chi}$.\\
		 							   If   $A=0$  we get from (\ref{r23}) that $\chi(z_0)\chi(xy)=0$, which is a contradiction because $\chi(z_0) \neq0$. So $A\neq 0$  and therefore from (\ref{r23}) and
		 							  by using Lemma \ref{ind} (b)   we get
		 							 \begin{equation*}
		 							 \label{xx}
		 							\alpha \chi(z_0)+ (\alpha c+\delta)A(z_0)\chi(z_0) =0\;\; and \;\; (\alpha c +\delta)\chi(z_0)  - 1=0,
		 							  \end{equation*}
		 							 from which  we read that  $\alpha c \neq  -\delta$ (because otherwise we get $\alpha =0 $ and $1=0)$, so we conclude that
		 							   $$\chi(z_0)=\dfrac{1}{\alpha c+\delta} \;\;\;  and \; \;\; A(z_0)=-\dfrac{\alpha}{\alpha c+\delta }.$$  This proves the forms in family (7).
		 							
		 							   	Conversely,  simple computations prove that the formulas above for      $f$ and $g$   define  solutions of (\ref{6}).
		 							   	
		 							  Finally,  suppose that $S$ is topological semigroup and $f, g \in C(S)$. The  continuity of cases (1) and (2) are evident.  The continuity  of $\chi$ and $\chi \circ \sigma$ in case (3)  follows directly from the form of $g$ by applying \cite[Theorem 3.18]{g} since $c_1\neq0$ and  $\chi\neq \chi \circ \sigma$. The case (4)   can be treated as the case (3). For the case (5) we have
		 							  \begin{equation*}
		 							  f-\alpha g= b (\chi -\chi\circ \sigma).
		 							  \end{equation*}
		 					 Then the continuity of $\chi$ and $\chi \circ \sigma$ follows from \cite[Theorem 3.18]{g} because $ b\neq0$ and $\chi\neq \chi \circ \sigma$. For the case (6)  we have $g=\Psi_{A\chi,\rho}$ is  continuous, then $\rho \in C(P_{\chi})$ and $A\chi \in C(S\setminus I_{\chi})$    by restriction. As $\alpha(f-cg)=\chi$ we get $\chi \in C(S)$, so $A\in C(S\setminus I_{\chi})$ since $\chi(x) \neq0$ for all  $x\in S\setminus I_{\chi}$. The parts (7) and (8) can be treated as case (6). This completes the proof  of Theorem  \ref{t1}.
		 							 \end{proof}
		 							 \section{Solution of the functional equation (\ref{2})}
		 							 	 In this section we solve on semigroups  the functional equation (\ref{2}), by using the results obtained in Theorem \ref{t1}.
		 							  \begin{thm}
		 							  	Let  $\lambda\in \mathbb{C}^{*}$. The solutions $f,g :S\longrightarrow \mathbb{C}$ of the  Kannappan-sine subtraction law (\ref{2}) can be listed as follows.\\
		 							  	(1) $f=\lambda g$ and $g$ arbitrary.\\
		 							  	(2)  $S\neq S^2z_0$, $f=0$    and $g$ is any non-zero function   such that $g(S^2z_0)=0$.\\
		 							  	(3)   $S\neq S^2z_0$,   $f=\dfrac{1+\kappa\lambda}{ \kappa}g$, where  $g$ is any non-zero function  such that $g(S^2z_0)=0$, and $\kappa\in\mathbb{C^{*}}$ is a constant.\\
		 							  	(4) There exist  constants $\gamma, b \in \mathbb{C}^{*}, c  \in \mathbb{C} \setminus \{\pm 1\} $ and an  exponential function $\chi$ on
		 							  	$S$  with $\chi \neq \chi\circ\sigma$, $\chi(z_0)=\dfrac{-2b}{1+c}$ and $\chi\circ\sigma(z_0)=\dfrac{2b}{1-c}$ such that
		 							  	\begin{equation*}
		 							  	f=   \dfrac{\chi+\chi\circ\sigma}{2\gamma} +(2\gamma  \lambda b+c)\dfrac{\chi-\chi\circ\sigma}{2\gamma}     \hspace{0.4cm} and \hspace{0.4cm}  g=b(  \chi-\chi\circ\sigma).
		 							  	\end{equation*}
		 							  	(5) There exist  constants $\beta, b \in \mathbb{C}^{*},c  \in \mathbb{C}   $ and an exponential function $\chi$ on
		 							  	$S$ with
		 							  	$\chi \neq \chi\circ\sigma$, $\chi(z_0)=  \chi\circ\sigma(z_0)=1/\beta$ such that
		 							  	\begin{equation*}
		 							  f=   \lambda\dfrac{\chi+\chi\circ\sigma}{2\beta} +(2\beta b+\lambda c  )\dfrac{\chi-\chi\circ\sigma}{2\beta}  \hspace{0.3cm} and \hspace{0.3cm}  g=  \dfrac{\chi+\chi\circ\sigma}{2\beta}+ c\dfrac{\chi-\chi\circ\sigma}{2\beta}.
		 							  	\end{equation*}
		 							  	(6) There exist  constants $\alpha,\delta, b  \in \mathbb{C}^{*}, c \in \mathbb{C}  $    and an exponential  function
		 							  	$\chi$  on	$S$
		 							  	with  $ \alpha \neq\pm (2b \delta +\alpha c)$,  $\chi \neq \chi\circ\sigma$, $	\chi (z_0)=	\dfrac {2b }{\alpha (1+c)+2b\delta   } \;\; and \;\; 	 \chi\circ \sigma(z_0)=\dfrac {2b }{\alpha (c-1)+2b\delta   }$
		 							  	such that
		 							  	$$
		 							  f=   (\alpha+ \lambda)\dfrac{\chi+\chi\circ\sigma}{2\delta} +[c(\alpha+\lambda)+2b\beta]\dfrac{\chi-\chi\circ\sigma}{2\delta}$$  and $$ g=   \dfrac{\chi+\chi\circ\sigma}{2\delta} +c  \dfrac{\chi-\chi\circ\sigma}{2\delta}.
		 							  	$$
		 							  	(7) There exist  constants $\gamma,  c\in \mathbb{C}^{*} $, an exponential function $\chi$ on $S$,  a non-zero
		 							  	additive
		 							  	function $A: S \setminus  I_{\chi}\longrightarrow \mathbb{C}$ and a function $\rho: P_{\chi}\longrightarrow \mathbb{C}$ with $\alpha c+\delta\neq0$, $\chi\circ\sigma=\chi, A \circ\sigma=$
		 							  	$-A, \rho\circ\sigma=-\rho $ and
		 							  	$\chi(z_0)=A(z_0)=\dfrac{-1}{c}$ such that
		 							  	\begin{equation*}
		 							  	f= \gamma\chi+ (\gamma c+\lambda) \Psi_{A\chi,\rho} \hspace{0.4cm} and \hspace{0.4cm} g= \Psi_{A\chi,\rho}.
		 							  	\end{equation*}
		 							  	(8) There exist  constants $ \alpha,  c \in \mathbb{C}$ and $ \delta \in \mathbb{C}^{*}$, an exponential function $\chi$ on $S$,
		 							  	a non-zero
		 							  	additive
		 							  	function $A: S \setminus  I_{\chi}\longrightarrow \mathbb{C}$ and a function $\rho: P_{\chi}\longrightarrow \mathbb{C}$ such that $\alpha c+\delta\neq0$,  $\chi\circ\sigma$
		 							  	$=\chi,  A \circ\sigma=-A,  \rho\circ\sigma=-\rho $,
		 							  	$\chi(z_0)=\dfrac{1}{\alpha c +\delta}$ and $A(z_0)=-\dfrac{\alpha}{\alpha c +\delta}$ such that
		 							  	\begin{equation*}
		 							  	f=  \dfrac{\alpha+\lambda}{\delta}\chi+ \dfrac{c(\lambda+\alpha)+\delta)}{\delta} \Psi_{A\chi,\rho}      \hspace{0.4cm} and \hspace{0.4cm} g=   \dfrac{1}{\delta} \left( \chi+ c \Psi_{A\chi,\rho}\right).
		 							  	\end{equation*}
		 							  	Moreover if $S$ is a topological semigroup and $f,g \in C(S)$, then $ \chi, \chi\circ\sigma \in C(S), A \in C(S \setminus  I_{\chi})$ and $\rho \in C(P_{\chi})$.
		 							  \end{thm}
		 							  \begin{proof}
		 							   The functional equation (\ref{2}) can  be rewritten as \begin{equation*}
		 							   \label{add}
		 							  	\left(  f-\lambda g\right) (x\sigma(y)z_0)=\left(  f-\lambda g\right)(x)g(y)-\left( f-\lambda g\right)(y)g(x)
		 							  	, \;x,y \in S,
		 							   \end{equation*}
		 							    and from this      we read that the pair $ \left(  f-\lambda g,  g\right) $ satisfies Kannappan-sine subtraction law (\ref{6}).  So according to Theorem \ref{t1} we get the following possibilities: \\
		 							    (1) $ f-\lambda g    =0$ and $g$ arbitrary. This gives  $g = \lambda f$ and then we get the solution part (1).\\
		 							    (2) $S\neq S^2z_0 $, $\left(  f- \lambda g\right)$ is any non-zero function such that $\left(  f- \lambda g\right)(S^2z_0)=0$  and $g=\kappa\left(  f-\lambda g\right)$ for some $\kappa\in\mathbb{C}$. This gives \begin{equation}
		 							    \label{89}g=  f=0 \;on \;S^2z_0 \;\;    and  \;\; (1+\kappa \lambda)g= {\kappa} f.\end{equation}
		 							
		 							    If $\kappa=0$ then from (\ref{89}) we get $g=0$,    and then $f$ is any non-zero function   such that $f(S^2z_0)=0$. This proves solution family (2).
		 							
		 							     If $\kappa\neq0$ we obtain from  (\ref{89})    that $f=\dfrac{1+\kappa\lambda}{ \kappa}g$ with $g$ is any non-zero function  such that $g(S^2z_0)=0$. Then we get solution in part (3).
		 							    \\
		 							    (3) There exist  constants $\gamma, b \in \mathbb{C}^{*},c \in \mathbb{C} \setminus \{\pm 1\} $ and an exponential function $\chi$ on
		 							    $S$  with $\chi \neq \chi\circ\sigma$, $\chi(z_0)=\dfrac{-2b}{1+c}$ and $\chi\circ\sigma(z_0)=\dfrac{2b }{1-c }$ such that
		 							    \begin{equation*}
		 							     f-{\lambda}g=  \dfrac{\chi+\chi\circ\sigma}{2\gamma} +c \dfrac{\chi-\chi\circ\sigma}{2\gamma}   \hspace{0.6cm} and \hspace{0.6cm}  g=b( \chi-\chi\circ\sigma),
		 							    \end{equation*}
		 							    then we get
		 							   \begin{equation*}
		 							    f=    \dfrac{\chi+\chi\circ\sigma}{2\gamma} +c \dfrac{\chi-\chi\circ\sigma}{2\gamma}     +\lambda g =    \dfrac{\chi+\chi\circ\sigma}{2\gamma} +(2\gamma  \lambda b+c)\dfrac{\chi-\chi\circ\sigma}{2\gamma}  .
		 							   \end{equation*}
So we get the solution in part (4).\\
		 							 The cases (4)-(7) of Theorem  \ref{t1}  can  be treated   as case  (2) below, and in this context  we show   that we get the solutions in parts (5)-(9) respectively.
		 							  	Conversely,  simple computations prove that the formulas above for      $f$ and $g$   define  solutions of (\ref{2}).
		 							  The topological statement can be treated as in Theorem \ref{t1}.
		 							  \end{proof}
		 						
		\section{ An example}
		In this section we give an example to illustrate the results obtained in Theorem \ref{t1}.
		\begin{ex}
			 We solve the Kannappan-sine subtraction (\ref{6}) on a semigroup $S$, which is neither a  group nor a monoid.
			
			We consider the semigroup
			 $S:=\mathbb{N}\times\mathbb{N}$ equipped with addition as composition rule, and discrete topology. $S$ is  a commutative semigroup which satisfies $S\neq S^2z_0$ for some fixed element $z_0=(s_0, t_0) $ in $S$, where $s_0,t_0\in \mathbb{N}$. Here we have
			 \begin{eqnarray*}
			     S^2z_0&=&\{ (x_1,y_1)+(x_2,y_2)+(s_0,t_0) | x_1,x_2,y_1,y_2 \in \mathbb{N} \}\\
			     & =&\{ (n,m) | n\geq s_0+2, m\geq t_0+2\},
			 \end{eqnarray*}
			 	and then
			 	 \begin{equation*}
			   S\setminus	S^2z_0= (\{n\in \mathbb{N}| 1\leq n\leq 1+s_0\} \times \mathbb{N}) \cup (\mathbb{N} \times \{m\in \mathbb{N}| 1\leq m\leq 1+s_0\} ).
			 	   \end{equation*}
			  We define
			\begin{equation*}
			\sigma (x,y):=(y,x),  \; x,y \in \mathbb{N}.
			\end{equation*}
			We indicate here  the corresponding continuous exponentials and additive functions on $S$.
			From \cite[Example 8.3]{q} we get that the continuous exponentials  $\chi:S \longrightarrow \mathbb{C}$ are the functions having the form
			\begin{equation}
			\label{88}
			\chi_{a_1 ,a_2}(x,y)= a_1^xa_2^y,\;   \text{for all} \; x,y \in \mathbb{N},
			\end{equation}
			 where $a_1 ,a_2\in \mathbb{C}^{*}$, and
			the additive functions $A:S \longrightarrow \mathbb{C}$ are given by
				\begin{center}
					$A_{b_1 ,b_2}(x,y)= {b_1}x + {b_2} y$,\;   for all $x,y \in \mathbb{N}$,
				\end{center}
				where $b_1 ,b_2\in \mathbb{C}^{*}$.

			The condition $\chi_{a_1 ,a_2}\circ \sigma =\chi_{a_1 ,a_2}$ becomes $  a_1^xa_2^y=  a_1^ya_2^x$, which implies that $a_1=a_2$, and the condition $A_{b_1 ,b_2}\circ \sigma=-A_{b_1 ,b_2}$ reduces to ${b_1}x + {b_2} y=-({b_2}x + {b_1} y)$, which gives  ${b_2}=-{b_1}$.
				
				Since the exponential defined in (\ref{88}) vanishes nowhere  $(I_{\chi}= \emptyset)$, then the corresponding function $\Psi_{A\chi, \rho}$  reduces to $\Psi_{A\chi}=A_{b_1 ,-b_1}\chi_{a_1 ,a_1} $.
				
				For $s_{0}\neq t_{0}$, the constraint $\chi\circ \sigma(z_0)  =\chi (z_0)$, in  the solution family (4) of Theorem \ref{t1}, implies that $a_1=a_2$, which contradicts the fact that $\chi  \circ \sigma \neq\chi  $, so this solution family does not occur when $s_{0}\neq t_{0}$.
				
					Finally, by help of Theorem \ref{t1}, for $s_{0}\neq t_{0}$, the continuous solutions $f,g :\mathbb{N}\times \mathbb{N} \longrightarrow \mathbb{C}$ of the Kannappan-sine subtraction law (\ref{6}), namely
				\begin{equation*}
				f(x_1+y_2+s_0, x_2+y_1+t_0)=f(x_1,y_1)g(x_2,y_2)-f(x_2,y_2)g(x_1,y_1),
				\end{equation*}
			for all  $x_1,x_2,y_1,y_2 \in \mathbb{N}$,	are the following pairs of functions:\\
				(1) $f=0$ and $g$ arbitrary.\\
					(2) $f$ is any non-zero function such that $f(S^2z_0)=0$  and $g=c f$, where $c\in\mathbb{C}$ is a constant.
					\\
					(3) There exist  constants $\gamma, b,a_1,a_2 \in \mathbb{C}^{*},c \in \mathbb{C} \setminus \{\pm 1\} $ with $a_{1}\neq a_{2}$, $a_1^{s_0}a_2^{t_0}= \dfrac{ -2b}{1+c}$ and $a_2^{s_0}a_1^{t_0}=\dfrac{2b}{1-c}$ such that
					\begin{equation*}
					f(x_1,y_1)=  \dfrac{ a_1^{x_1}a_2^{y_1}+ a_1^{y_1}a_2^{x_1}}{2\gamma} +c \dfrac{a_1^{x_1}a_2^{y_1}- a_1^{y_1}a_2^{x_1}}{2\gamma}    \;\;and
					\end{equation*}
					\begin{equation*}
					  g(x_1,y_1)=b(a_1^{x_1}a_2^{y_1}- a_1^{y_1}a_2^{x_1} ).
					\end{equation*}
					(4) There exist  constants $\alpha,\delta, b, a_1,a_2 \in \mathbb{C}^{*}, c\in \mathbb{C}  $
					with $a_{1}\neq a_{2}$, $\alpha\neq\pm (2b\delta +\alpha c)
					$,  $ a_1^{s_0}a_2^{t_0}= \dfrac{ 2b}{\alpha(c+1)+2b\delta}$  and $a_1^{t_0}a_2^{s_0}=\dfrac{2b}{\alpha(c-1)+2b\delta}$
					such that
					\begin{equation*}
					f(x_1,y_1)=\alpha   \dfrac{a_1^{x_1}a_2^{y_1}+ a_1^{y_1}a_2^{x_1}}{2\delta} +(\alpha c+2b\delta)\dfrac{a_1^{x_1}a_2^{y_1}- a_1^{y_1}a_2^{x_1}}{2\delta} \;  and
					\end{equation*}
						\begin{equation*}
						  g(x_1,y_1)=   \dfrac{a_1^{x_1}a_2^{y_1}+ a_1^{y_1}a_2^{x_1}}{2\delta} +c  \dfrac{a_1^{x_1}a_2^{y_1}- a_1^{y_1}a_2^{x_1}}{2\delta}.
					\end{equation*}
					(5) There exist  constants $\alpha,  c, a_1, b_1 \in\mathbb{C}^{*} $   with
				 	$ {a_1}^{s_0+t_0}=b_1(s_0-t_0)=\dfrac{-1}{c}$ such that
					\begin{equation*}
					f(x_1,y_1)= \alpha{a_1}^{x_1+y_1}[ 1 +c b_1(x_1-y_1)]  \;\; and \;\; g(x_1,y_1)= b_1a_1^{x_1+y_1}(x_1-y_1).
					\end{equation*}
					(6) There exist  constants $ \alpha,\delta,a_1,b_1 \in \mathbb{C}^{*},  c \in \mathbb{C} \setminus \{-\delta/\alpha \}  $ with
					${a_1}^{s_0+t_0}=\dfrac{1}{\alpha c+\delta }$ and $ b_1(s_0-t_0)=-\dfrac{\alpha}{\alpha c+\delta}$ such that
					\begin{equation*}
					f(x_1,y_1)=   \dfrac{1}{\delta}{a_1}^{x_1+y_1} [ \alpha+ b_1(\alpha c+\delta) (x_1-y_1)]  \;\;and
					\end{equation*}
						\begin{equation*}
					 g(x_1,y_1)=   \dfrac{1}{\delta} {a_1}^{x_1+y_1}[  1+ b_1c (x_1-y_1)].
					\end{equation*}

		\end{ex}

\textbf{Conflict of interest.} On behalf of all authors, the corresponding author states that there is no conflict of interest.

\end{document}